\documentclass[12pt,oneside,reqno]{amsart}

\usepackage{amssymb}
\usepackage[active]{srcltx}
\usepackage{a4wide}
\usepackage{amsmath,amsthm}
\usepackage{amssymb}
\usepackage{amstext}
\everymath{\displaystyle}
\usepackage{cite}
\usepackage{epsfig}
\usepackage{enumerate}
\usepackage{color}
 \usepackage{setspace} 
\usepackage{amssymb}
\newtheorem{theorem}{Theorem}[section]

\theoremstyle{definition}

\theoremstyle{remark}

\usepackage[linkcolor=blue, urlcolor=blue, citecolor=blue,
colorlinks, bookmarks]{hyperref}
\usepackage{hyperref}

\newcommand{\be}{\begin{equation}}
\newcommand{\ee}{\end{equation}}

\begin{document}

\title[]{Some results on Continuous dependence of fractal functions on the Sierpi\'nski gasket}



\author{*Vishal Agrawal}
\address{Department of Mathematical Sciences, IIT(BHU), Varanasi, India 221005 }
\email{*vishal.agrawal1992@gmail.com}

\author{Ajay Prajapati}
\address{Department of Computer science, Banaras Hindu University, Varanasi, India 221005}

\email{ajaypraja640@gmail.com}

\author{Abhilash Sahu}
\address{Department of Mathematics
Indian Institute of Technology Guwahati,
Guwahati, Assam, 
India, 781039}

\email{sahu.abhilash16@gmail.com}

\author{Tanmoy Som}
\address{Department of Mathematical Sciences, IIT(BHU), Varanasi, India 221005\\}
\email{tsom.apm@iitbhu.ac.in}





 

\begin{abstract}
In this article, we show that $\alpha$-fractal functions defined on Sierpi\'nski gasket (denoted by $\triangle$) depend continuously on the parameters involved in the construction. In the latter part of this article, the continuous dependence of parameters on $\alpha$-fractal functions defined on $\triangle$ is shown graphically. 
\end{abstract}

\maketitle



\section{Introduction}\label{section 1} 
In the field of Numerical Analysis, the computational implications of interpolation have been a major concern for many years. The theory of interpolation has evolved throughout the development of classical approximation theory. Various interpolation techniques are used in Numerical Analysis and Classical Approximation Theory, and they are based on polynomial, trigonometric, spline, and rational functions. Based on the underlying idea of the model under investigation, these techniques can be applied to a specific data set. While nonrecursive interpolation techniques in the literature almost always produce smooth interpolants, it should be emphasized that they are not all recursive. There are a great number of real-world phenomena and experimental signals that are confusing, and sometimes their traces appear smooth. Due to their non-differentiability and complexity, a simple mathematical framework may not adequately describe their smallest geometric complexity. In order to produce an interpolant with a more complex geometric structure, it is necessary to develop a novel interpolation approach.  Univariate real-valued interpolation functions constructed on a compact interval in $\mathbb{R}$ were first proposed by Barnsley \cite{MF1}. These are referred to as Fractal Interpolation Functions (FIFs for short), and their construction is based on the Iterated Function System (IFS) theory \cite{Hut}. In \cite{MF1}, the groundbreaking research on fractal interpolation has attracted a lot of interest in the literature and is currently going strong. The author of \cite{M2} demonstrated that the FIFs theory can be used to construct a family of continuous functions with fractal properties from a given continuous function.

 A number of studies have exposed several significant characteristics of FIFs, including their smoothness, stability, one-sided approximation property and constraint approximation property, as well as their box dimension and Hausdorff dimension. There are several research articles available on various types of FIFs, see, for instance, \cite{SS2, SS51, JVC1,VC2, Mnew2, coal}. 
Numerous studies reveal several significant aspects of FIFs, such as their smoothness, stability, one-sided approximation property and constrained approximation properties, as well as box dimension and Hausdorff dimension of their graphs.
Celik et al. \cite{Celik} expanded the notion of FIFs to incorporate interpolation of a data set on $\triangle$. Following this article, Ruan \cite{Ruan3} developed FIFs on post critically finite self-similar sets. Kigami \cite{Kig} is credited with having introduced and analysed these sets. On $\triangle$, Ri and Ruan \cite{Ruan4} defined several fundamental characteristics of a space of FIFs. The fractal dimensions of FIFs defined on various domains have been thoroughly investigated in numerous articles, see, for instance, \cite{Chana, SS5, Fal, Verma21, Pnd, Pr1, SP, Manuj, VPV2  ,VPV1, VM1}. Recently, Mohapatra et al. \cite{MNSV} introduced a concept that generalised the notions of the Kannan map and contraction. In \cite{PV11}, Prasad and Verma  have constructed the FIFs on the product of two Sierpi\'nski gaskets. \par
We denote the space of all the real-valued continuous functions defined on $\triangle$ by  $C(\triangle, \mathbb{R})$ and graph of $f$ by $graph(f)$ throughout this paper. 
   
 \section{Fractal interpolation function on the Sierpi\'nski gasket}
 We begin by providing a brief overview of the relevant concepts and an introduction to $\triangle$. The reader may refer to \cite{LP, Ri2, stri, vermabv} for further information. We begin by recalling an established $\triangle$ construction based on IFS.     
Consider a set $V_0=\left\{x_1, x_2, x_3\right\}$ such that points in $V_0$ have equal distances from each other. Corresponding to each point of $V_{0}$, define the contraction map $u_{i}$ on $\mathbb{R}^2$ as follows:

$$u_i(t)=\frac{1}{2}\left(t+x_i\right),$$ 
where $i=1,2,3$. Then, three contraction maps together with the plane  constitute an IFS, which produces $\triangle$ as an attractor, i.e., 
$$
\triangle= \bigcup_{i=1}^{3}  u_i(\triangle). 
$$
Define $V_1$ by $V_1=\left\{x_1, x_2, x_3, u_1\left(x_2\right), u_2\left(x_3\right), u_3\left(x_1\right)\right\}$. Let us consider a continuous function $f: \triangle \rightarrow \mathbb{R}$.  Let $U=\triangle \times \mathbb{R}$ and define maps $H_i: U \rightarrow U$ by
$$
H_i(t, x)=\left(u_i(t), M_i(t, x)\right), i=1,2,3,
$$ 
 
where $M_i(t, x): \triangle \times \mathbb{R} \rightarrow \mathbb{R}$ is a contraction map in the last variable, that is,
$$
\left|M_i(., x)-M_i\left(., x^{\prime}\right)\right| \leq c\left|x-x^{\prime}\right|
$$
with  $M_i\left(p_j, f\left(p_j\right)\right)=f\left(u_i\left(p_j\right)\right)$. In particular, we take
$$
M_i(t, x)=\alpha_i x+f\left(u_i(t)\right)-\alpha_i b(t),
$$
where $b \in C(\triangle, \mathbb{R})$ is base function and $f \in C(\triangle, \mathbb{R})$ original function  such that $b\left(p\right)=f\left(p\right)~ \forall~ p\in V_0$, and scale vector $\alpha_i \in \mathbb{R}$ with $\left|\alpha_i\right|_{\infty}<1$. Thus, We have an IFS $\mathcal{J}:= \left\{U, H_i: i=1,2,3\right\}$.\par

\begin{theorem}\cite{EPJS}
Let $f\in C(\triangle, \mathbb{R})$, and $b \in C(\triangle, \mathbb{R})$. Then, above defined IFS  $\mathcal{J}$ has a unique attractor  $graph(f^\alpha)$. The set $graph(f^\alpha)$ is the graph of a continuous function $f^\alpha: \triangle \rightarrow \mathbb{R}$, which satisfies $\left.f^\alpha\right|_{V_1}=\left.f\right|_{V_1}$. Furthermore, we have the following functional equation
\begin{equation}
  f^\alpha(t)=f(t)+\alpha_i\left(f^\alpha-b\right)\left(u_i^{-1}(t)\right)~ \forall ~t \in u_i(\triangle), i \in\{1,2,3\}.  
\end{equation}
\end{theorem}

\section{Continuous dependence on parameter $\alpha$ AND $b$}
Let $f \in C(\triangle, \mathbb{R})$. Define a map $W$ from $S$ to $C(\triangle, \mathbb{R})$ by
$$
W(\alpha)=f_{b}^\alpha,
$$
where $S=\left\{\alpha \in \mathbb{R}^3:|\alpha|_{\infty} \leq r<1\right.$ and $r$ is a fixed number$\}$ and $f_b^\alpha$ is $\alpha$-fractal function associated with $f$ with respect to $b$ and the scale vector $\alpha$.

\begin{theorem}
 For fixed $f \in C(\triangle, \mathbb{R})$ and for suitable $b \in C(\triangle, \mathbb{R})$, the map $W: S \rightarrow C(\triangle, \mathbb{R})$ is continuous.
\end{theorem}
\begin{proof}
The fixed point theory says that for a fixed a scale vector $\alpha$, and $b$, the map $f_b^\alpha$ is unique. Further, being fixed point of RB- operator, $f_b^\alpha$ satisfies the functional equation:
$$
f_b^\alpha(t)=f(t)+\alpha_i\left(f_b^\alpha-b\right) \circ u_i^{-1}(t), ~\forall~ t \in u_i(\triangle), i \in\{1,2,3\} .
$$
It is obvious that $W$ is well defined. Let $\alpha \in S$, then from the above functional equation, we have
$$
W(\alpha)(t)=f_b^\alpha(t)=f(t)+\alpha_i\left(f_b^\alpha-b\right) \circ u_i^{-1}(t), ~\forall ~t \in u_i(\triangle), i \in\{1,2,3\}
$$
and for $\beta \in S$,
$$
W(\beta)(t)=f_b^\beta(t)=f(t)+\beta_i\left(f_b^\beta-b\right) \circ u_i^{-1}(t), ~\forall~t \in u_i(\triangle), i \in\{1,2,3\}.
$$
We shall show that $W$ is continuous at $\alpha$. For this, subtract one from other of the above two equations, for $t \in u_i(\triangle)$, we have
\begin{equation}
    \begin{aligned}
f_b^\alpha(t)-f_b^\beta(t) &=\left[f(t)+\alpha_i\left(f_b^\alpha-b\right) \circ u_i^{-1}(t)\right]-\left[f(t)+\beta_i\left(f_b^\beta-b\right) \circ u_i^{-1}(t)\right] \\
&=\left[\alpha_i f_b^\alpha\left(u_i^{-1}(t)\right)-\beta_i f_b^\beta\left(u_i^{-1}(t)\right)\right]+\left(\beta_i-\alpha_i\right) b \circ u_i^{-1}(t) \\
&=\left[\alpha_i f_b^\alpha-\beta_i f_b^\alpha+\beta_i f_b^\alpha-\beta_i f_b^\beta\right] \circ u_i^{-1}(t)+\left(\beta_i-\alpha_i\right) b \circ u_i^{-1}(t) \\
&=\left[\left(\alpha_i-\beta_i\right) f_b^\alpha+\beta_i\left(f_b^\alpha-f_b^\beta\right)\right] \circ u_i^{-1}(t)+\left(\beta_i-\alpha_i\right) b \circ u_i^{-1}(t).
\end{aligned}
\end{equation}
Now, using triangle inequality and definition of uniform norm, we have
\begin{equation}
    \begin{aligned}
\left|f_b^\alpha(t)-f_b^\beta(t)\right| & \leq\left|\left[\left(\alpha_i-\beta_i\right) f_b^\alpha+\beta_i\left(f_b^\alpha-f_b^\beta\right)\right] \circ u_i^{-1}(t)\right|+\left|\left(\beta_i-\alpha_i\right) b \circ u_i^{-1}(t)\right| \\
& \leq|\alpha-\beta|_{\infty}\left\|f_b^\alpha\right\|_{\infty}+|\beta|_{\infty}\left\|f_b^\alpha-f_b^\beta\right\|_{\infty}+|\beta-\alpha|_{\infty}\|b\|_{\infty} \\
&=|\alpha-\beta|_{\infty}\left(\left\|f_b^\alpha\right\|_{\infty}+\|b\|_{\infty}\right)+|\beta|_{\infty}\left\|f_b^\alpha-f_b^\beta\right\|_{\infty} .
\end{aligned}
\end{equation}
It follows that, for all $t \in \triangle$, we get
\begin{equation}
    \left|f_b^\alpha(t)-f_b^\beta(t)\right| \leq|\alpha-\beta|_{\infty}\left(\left\|f_b^\alpha\right\|_{\infty}+\|b\|_{\infty}\right)+|\beta|_{\infty}\left\|f_b^\alpha-f_b^\beta\right\|_{\infty}.
\end{equation}
The above implies that
\begin{equation}
    \left\|f_b^\alpha-f_b^\beta\right\|_{\infty} \leq|\alpha-\beta|_{\infty}\left(\left\|f_b^\alpha\right\|_{\infty}+\|b\|_{\infty}\right)+|\beta|_{\infty}\left\|f_b^\alpha-f_b^\beta\right\|_{\infty}.
\end{equation}
Using $1-|\beta|_{\infty} \geq 1-r$, finally we have

\begin{equation}
\|W(\alpha)-W(\beta)\|_{\infty}=\|f_b^\alpha-f_b^\beta\|_{\infty} \leq \frac{|\alpha-\beta|_{\infty}}{1-r}\left(\left\|f_b^\alpha\right\|_{\infty}+\|b\|_{\infty}\right).
\end{equation}
Since $\alpha$ is fixed and $\left\|f_b^\alpha\right\|_{\infty}$ is bounded, we have $W$ is continuous at $\alpha$. Since $\alpha$ was taken arbitrarily, hence, $W$ is continuous on $S$.
\end{proof}
\begin{theorem}
Let $f \in C(\triangle, \mathbb{R})$ and scale vector $\alpha \in \mathbb{R}^3$ with $|\alpha|_{\infty}<1$ and $X_f=\left\{b \in C(\triangle, \mathbb{R}):\left.b\right|_{V_0}=\left.f\right|_{V_0}\right\}$. Then the map $T: X_f \rightarrow C(\triangle, \mathbb{R})$ defined by $T(b)=f_b^\alpha$ is Lipschitz continuous.
\end{theorem}
\begin{proof}
 We know that for a scale vector $\alpha$ and a suitable function $b: \triangle \rightarrow \mathbb{R}$, the function $f_b^\alpha$ is unique. Further, being fixed point of $\mathrm{RB}$ - operator, $f_b^\alpha$ satisfies the functional equation:
$$
f_b^\alpha(t)=f(t)+\alpha_i\left(f_b^\alpha-b\right) \circ u_i^{-1}(t),~ \forall~ t \in u_i(\triangle), i \in\{1,2,3\} .
$$
It is obvious that $T$ is well defined. Let $b, c \in X_f$ then from the above functional equation, we have
$$
T(b)(t)=f_b^\alpha(t)=f(t)+\alpha_i\left(f_b^\alpha-b\right) \circ u_i^{-1}(t), ~\forall~ t \in u_i(\triangle), i \in\{1,2,3\}.
$$
and
$$
T(c)(t)=f_c^\alpha(t)=f(t)+\alpha_i\left(f_c^\alpha-c\right) \circ u_i^{-1}(t),~ \forall~ t \in u_i(\triangle), i \in\{1,2,3\}.
$$
On subtracting one from other of the above two equations, we get for $t \in u_i(\triangle)$
$$\begin{aligned}
f_b^\alpha(t)-f_c^\alpha(t) &=\left[f(t)+\alpha_i\left(f_b^\alpha-b\right) \circ u_i^{-1}(t)\right]-\left[f(t)+\alpha_i\left(f_c^\alpha-c\right) \circ u_i^{-1}(t)\right] \\
&=\alpha_i\left(f_b^\alpha-f_c^\alpha\right) \circ u_i^{-1}(t)+\alpha_i(c-b) \circ u_i^{-1}(t).
\end{aligned}$$
Now using triangle inequality and definition of uniform norm, we have
$$
\begin{aligned}
\left|f_b^\alpha(t)-f_c^\alpha(t)\right| &=\left|\alpha_i\left(f_b^\alpha-f_c^\alpha\right) \circ u_i^{-1}(t)+\alpha_i(c-b) \circ u_i^{-1}(t)\right| \\
& \leq\left|\alpha_i\left(f_b^\alpha-f_c^\alpha\right) \circ u_i^{-1}(t)\right|+\left|\alpha_i(c-b) \circ u_i^{-1}(t)\right| \\
& \leq|\alpha|_{\infty}\left\|f_b^\alpha-f_c^\alpha\right\|_{\infty}+|\alpha|_{\infty}\|c-b\|_{\infty}.
\end{aligned}
$$
The above inquality holds for all $t \in \triangle$, therefore, we write
$$
\left\|f_b^\alpha-f_c^\alpha\right\|_{\infty} \leq|\alpha|_{\infty}\left\|f_b^\alpha-f_c^\alpha\right\|_{\infty}+|\alpha|_{\infty}\|c-b\|_{\infty}
.$$
This can be recasted as, $\left\|f_b^\alpha-f_c^\alpha\right\|_{\infty} \leq \frac{|\alpha|_{\infty}}{1-|\alpha|_{\infty}}\|b-c\|_{\infty}$. It follows that 
$$
\| T(b)- T(c)\|_{\infty} \leq \frac{|\alpha|_{\infty}}{1-|\alpha|_{\infty}}\| b-c \|_{\infty},
$$
which shows that $T$ is a Lipschitz continuous map with Lipschitz constant $\frac{|\alpha|_{\infty}}{1-|\alpha|_{\infty}}$.
\end{proof}

Next, we plot the $graph(f_{b}^\alpha)$ for different values of parameters, that is, original function, base function and scale vector. One can easily identify the variation in these graphs by changing the values of parameters. Hence, $graph(f_{b}^\alpha)$ depends on these parameters. 
\begin{figure}[h!]
\begin{minipage}{0.5\textwidth}
                       \includegraphics[width=1.0\linewidth]{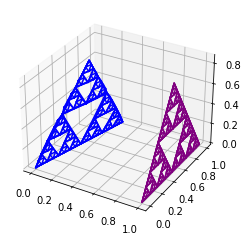}
                       {\hspace*{1cm}$graph(f_{b}^\alpha)$ at $\alpha = 0.1$}
                       \end{minipage}\hspace*{0.3cm}
                       \begin{minipage}{0.5\textwidth}
                \includegraphics[width=1.0\linewidth]{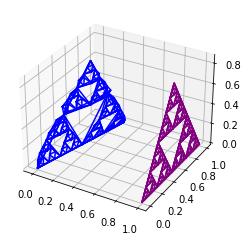}
                       {\hspace*{3cm}$graph(f_{b}^\alpha)$ at $\alpha = 0.3$}
                       \end{minipage}
                       \begin{minipage}{0.5\textwidth}
                       \includegraphics[width=1.0\linewidth]{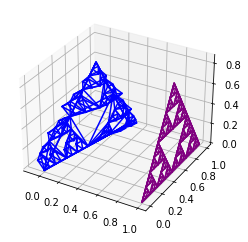}
                       {\hspace*{3cm}$graph(f_{b}^\alpha)$ at $\alpha = 0.6$}
                       \end{minipage}\hspace*{0.3cm}
                       \begin{minipage}{0.5\textwidth}
                       \includegraphics[width=1.0\linewidth]{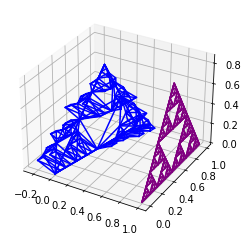}
                       {\hspace*{3cm}$graph(f_{b}^\alpha)$ at $\alpha = 0.9$}
                       \end{minipage}
                       \caption{$graph(f_{b}^\alpha)$ for the various values of the  $\alpha$, where $b(x,y)= \frac{x}{4}+\frac{y}{9}-1.3y(x-0.5)$,     $f(x,y)=\frac{x}{4}+\frac{y}{9}$.}
                       \end{figure}

\begin{figure}[h!]
\begin{minipage}{0.5\textwidth}
                       \includegraphics[width=1.0\linewidth]{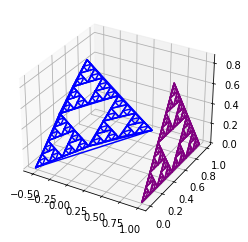}
                       {\hspace*{1cm}$graph(f_{b}^\alpha)$ at $\alpha = 0.1$}
                       \end{minipage}\hspace*{0.3cm}
                       \begin{minipage}{0.5\textwidth}
                \includegraphics[width=1.0\linewidth]{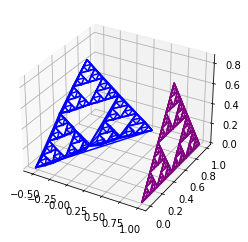}
                       {\hspace*{3cm}$graph(f_{b}^\alpha)$ at $\alpha = 0.3$}
                       \end{minipage}
                       \begin{minipage}{0.5\textwidth}
                       \includegraphics[width=1.0\linewidth]{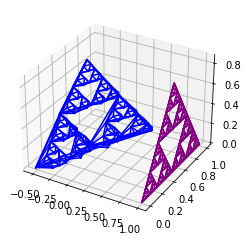}
                       {\hspace*{3cm}$graph(f_{b}^\alpha)$ at $\alpha = 0.6$}
                       \end{minipage}\hspace*{0.3cm}
                       \begin{minipage}{0.5\textwidth}
                       \includegraphics[width=1.0\linewidth]{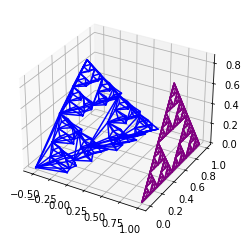}
                       {\hspace*{3cm}$graph(f_{b}^\alpha)$ at $\alpha = 0.9$}
                       \end{minipage}
                       \caption{$graph(f_{b}^\alpha)$ for the various values of the  $\alpha$, where $b(x,y)= sin(x+ 3.7)+1.3x-x^{2}y+0.866x^{2}+xy-0.866x$,     $f(x,y)= sin(x+ 3.7)+1.3x$.}
                       \end{figure}

\begin{figure}[h!]
\begin{minipage}{0.5\textwidth}
                       \includegraphics[width=1.0\linewidth]{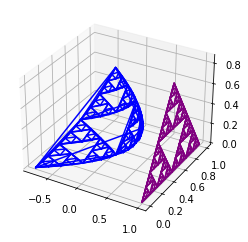}
                       {\hspace*{1cm}$graph(f_{b}^\alpha)$ at $\alpha = 0.1$}
                       \end{minipage}\hspace*{0.3cm}
                       \begin{minipage}{0.5\textwidth}
                \includegraphics[width=1.0\linewidth]{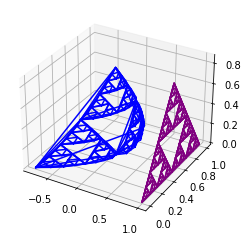}
                       {\hspace*{3cm}$graph(f_{b}^\alpha)$ at $\alpha = 0.3$}
                       \end{minipage}
                       \begin{minipage}{0.5\textwidth}
                       \includegraphics[width=1.0\linewidth]{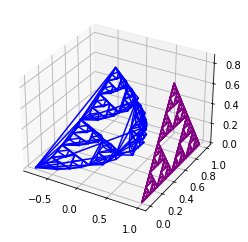}
                       {\hspace*{3cm}$graph(f_{b}^\alpha)$ at $\alpha = 0.6$}
                       \end{minipage}\hspace*{0.3cm}
                       \begin{minipage}{0.5\textwidth}
                       \includegraphics[width=1.0\linewidth]{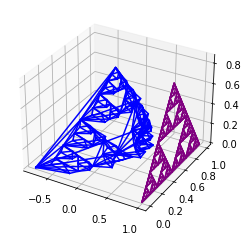}
                       {\hspace*{3cm}$graph(f_{b}^\alpha)$ at $\alpha = 0.9$}
                       \end{minipage}
                       \caption{$graph(f_{b}^\alpha)$ for the various values of the  $\alpha$, where $b(x,y)= cos(2x +5)+ sin(x + 2.7)-1.5 + 1.3x -x^{2}y+0.866x^{2}+xy-0.866x$,     $f(x,y)= cos(2x+5)+ sin(x + 2.7)-1.5 + 1.3x$.}
                       \end{figure} 
\clearpage
\begin{figure}[h!]
\begin{minipage}{0.5\textwidth}
                       \includegraphics[width=1.0\linewidth]{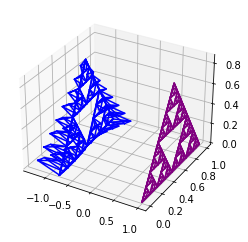}
                       {\hspace*{1cm}$graph(f_{b}^\alpha)$ at $\alpha = 0.1$}
                       \end{minipage}\hspace*{0.3cm}
                       \begin{minipage}{0.5\textwidth}
                \includegraphics[width=1.0\linewidth]{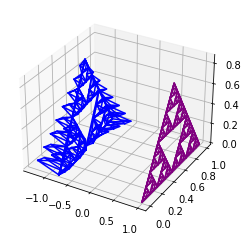}
                       {\hspace*{3cm}$graph(f_{b}^\alpha)$ at $\alpha = 0.3$}
                       \end{minipage}
                       \begin{minipage}{0.5\textwidth}
                       \includegraphics[width=1.0\linewidth]{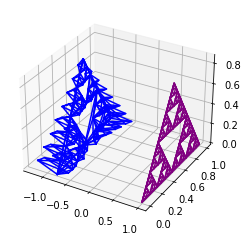}
                       {\hspace*{3cm}$graph(f_{b}^\alpha)$ at $\alpha = 0.6$}
                       \end{minipage}\hspace*{0.3cm}
                       \begin{minipage}{0.5\textwidth}
                       \includegraphics[width=1.0\linewidth]{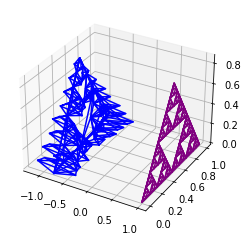}
                       {\hspace*{3cm}$graph(f_{b}^\alpha)$ at $\alpha = 0.9$}
                       \end{minipage}
                       \caption{$graph(f_{b}^\alpha)$ for the various values of the  $\alpha$, where $b(x,y)= cos(100x +5)+ sin(x + 2.7)-1.5 + 1.3x -x^{2}y+0.866x^{2}+xy-0.866x$,     $f(x,y)= cos(100x+5)+ sin(x + 2.7)-1.5 + 1.3x$.}
                       \end{figure}

 \section{Declaration}

\textbf{Funding.}  Not applicable. \\

\textbf{Conflicts of interest.} We do not have any conflict of interest.\\

\textbf{Availability of data and material.} Not applicable.\\

\textbf{Code availability.} Not applicable.\\

\textbf{Authors' contributions.} Each author contributed equally in this manuscript.

\bibliographystyle{amsplain}

\end{document}